%% file: pub_191210.tex
\documentclass[12pt,english]{article}
\usepackage{babel}
\usepackage[cp1250]{inputenc}
\usepackage[T1]{fontenc}
\usepackage{amsfonts}
\usepackage{amsmath}
\usepackage{amsthm}
\usepackage{mathrsfs}
\usepackage{hyperref}
\usepackage{enumerate}
\usepackage{graphicx}
\usepackage{pictex}
\usepackage{color}
\usepackage{loc191118}

\title{Existence and regularity result for Stokes system with special inlet/outlet condition}

\author{Kamil Wo\l os, Przemys\l aw Kosewski
\footnote{
Department of Mathematics and Information Sciences, Warsaw University of Technology, ul. Koszykowa 75,
00-662 Warsaw, Poland, E-mail addresses: kamwolos@gmail.com, P.Kosewski@mini.pw.edu.pl} }

\begin{document}
\maketitle

\begin{abstract}
Our aim is to analyse special type of boundary conditions, created to simulate flows like in cardiovascular and respiratory systems. Firstly, we will describe model of viscous, incompressible fluid in a domain consisting many inlets and outlets with open dissipative boundary conditions. The conditions are augmented by the inertia terms. We are posing additional constrains on a fluid motion by a volumetric flow rates or inlet/outlet pressure. Afterwards, we will define weak formulation of the problem and its motivation. Then, we will prove mathematical correctness of proposed conditions by properly modified Galerkin method. Also, we will prove existence of a solution and its uniqueness. 
\end{abstract}
\vspace{5mm}


\section{Introduction}
Numerical modelling of unsteady incompressible flows in the large domains with many branches is still a big challenge for both  mathematicians and engineers. Suppose that we are interested in the fluid flow simulation in domain, which consists of the bifurcation tree, where every branch is divided onto two sub-branches, [see pic. \ref{fig:dziedzina_existence}]. The system of branches has approximately 16 levels of bifurcation. This situation can appear for example in respiratory or cardiovascular systems. 

\begin{figure}
	\centering
	\resizebox{0.5\textwidth}{!}{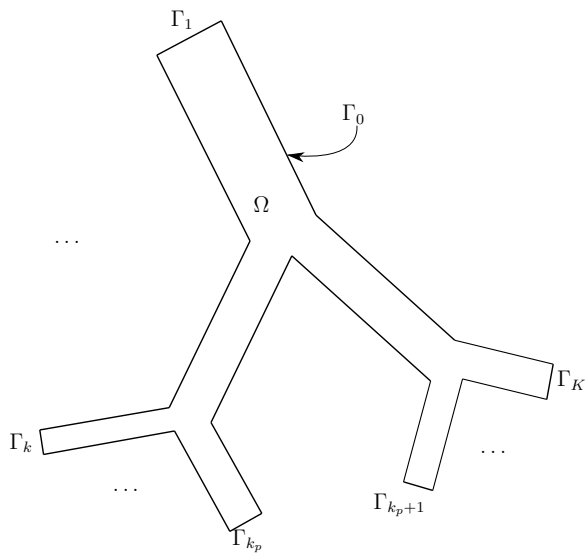}	
	\caption{The scheme of the domain.}
	\label{fig:dziedzina_existence}
\end{figure}

We can distinguish two major difficulties, that arise in this setting. First is geometric complexity of the domain. The whole system is too large from the numerical point of view, thus the computation of the velocity field in whole domain is out of reach. However, in many cases the flow simulation for whole domain is not necessary and it is enough, if we consider only a part of the system. Hence, the second problem is how to impose boundary conditions for the restricted domain. 

In the restricted domain problems we use, so called, artificial boundary conditions. They are ,,artificial'' in the sense, that these conditions are simply pipe sections separating the domain of interest from another component of the network. 

There are many approaches for this problem. In \cite{Heywood} the \textit{do-nothing} boundary conditions for network systems have been proposed. In \cite{Formaggia} have been described boundary conditions based on Lagrange multipliers. Many other approaches were also described in Maury's monograph \cite{Maury}. In this paper we will focus on a special type of the boundary conditions called open dissipative. The starting point is the paper of Szumbarski, \cite{Szumb}, where full description of the mechanical interpretation of the unsteady Stokes problem can be found, as well as numerical analysis. The boundary conditions from \cite{Szumb} are strictly related to general open disspiative conditions, which can be found in \cite{Maury}. 

The main purpose of this paper is to establish results concerning existence and uniqueness of weak solutions to the problem from \cite{Szumb}. Now we will briefly introduce the model. 

The main outline of the model is as follows. The domain consist of the rigid impermeable wall $\Gamma_0$ with no-slip boundary condition and the inlet/outlet sections $\Gamma_k, \, k = 1, \ldots, K$, where open/dissipative boundary conditions are imposed. Description of this boundary conditions can be found in \cite{Maury}. In general, we assume, that inlet/outlet is connected with outside world by the virtual pipe, where the Poiseuille's law is preserved. Then, the pressure difference between at the outlet and in the far field can be expressed as a linear dependence with the flux:
\begin{align*}
	R \int_\Gamma v \cdot n \, d \Gamma = \Pi - S,
\end{align*}
where $\Pi n = \nu \nabla u - p n$ and $p, v, \nu$ denote, respectively, pressure, velocity and dynamic viscosity of a fluid. The symbol $n$ denotes the external normal vector to the boundary. In our case the scalar function $S$ is given. 
 
In \cite{Szumb} imposed inlet/outlet are modified in a following way: 
\begin{align}
	\hspace{3mm} p n - \nu \nabla v \cdot n - n \left( \lambda_k + \gamma_k  \frac{d}{dt} \right) \left( v \cdot n \right) = S_k n \quad \text{na } \Gamma_k, \, k = 1, \ldots , K.
	\label{WB_B2}
\end{align}
The coefficients $\lbrace \lambda_k, \gamma_k \rbrace > 0$ are given for all $k = 1, \ldots, K$. We assume, that the scalar functions $\lbrace S_1(t), \ldots, S_K(t) \rbrace$ are also given. 

The physical interpretation is following. Suppose that given inlet/outlet is flat. It can be shown that 
\begin{align*}
	\int_{\Gamma_k} \tau \cdot (\nabla u \cdot n) d \Gamma = 0,
\end{align*}
where $\tau$ is perpendicular vector to $n$. If we integrate (\ref{WB_B2}) over $\Gamma_k$ and divide by $|\Gamma_k|$, it leads to he following equation 
\begin{align*}
	\overline{p_k} - S_k = |\Gamma_k|^{-1} \left( \lambda_k + \gamma_k \frac{d}{dt} \right) \int_{\Gamma_k} v \cdot n \, d \Gamma.
\end{align*}
The quantity $\overline{p_k} = |\Gamma_k|^{-1} \int_{\Gamma_k} p \, dS$ is average pressure on the inlet/outlet. We can interpret this situation as the difference between pressure on inlet/outlet and pressure in the far field. The difference is expressed as a sum of two components: static (which is proportional to flow rate $\int_{\Gamma_k} v \cdot n \, d \Gamma$) and dynamic (which is proportional to the rate of change of the flow rate). 
In \cite{Szumb} the author obtained numerical solution based on the splitting method. In this paper we give a proof that unsteady, incompressible Stokes equation: 
\begin{align}
	\begin{aligned}
		v_t - \nu \Delta v + \nabla p = f, \\
		\nabla \cdot v = 0,
	\end{aligned}
	\label{stokes_sys}
\end{align}
\begin{equation}
	v|_{\Gamma_0} = 0,
	\label{no_slip_cond}
\end{equation}
with boundary condition (\ref{WB_B2}) possesses a weak solution (see definition \ref{weak_sol_def}), which is unique.

\section{Notation}
In this section we will introduce notation for function spaces, that will be used in following sections. Firstly, let $V$ denote
\begin{align*}
V = \left \{
v \in H^1(\Omega): \nabla \cdot v = 0 \text{ in } \Omega, \, v|_{\Gamma_0} = 0
\right \}.
\end{align*}
Let $n_k$ be a normal unit vector to surface $\Gamma_k$. 
Next, let space $\VH$ be defined by 
\eqnsl{
\VH = \left \{
v \in H^2(\Omega) \cap V: \forall_{k \in \{ 1, \dots, K \}} \left( \nabla v \cdot n_k \right) \cdot \tau |_{\Gamma_k} =  0, \forall_{\tau: \tau \cdot n_k = 0}
\right \}
}{VH2}
Definition of space $ \VH $ (\ref{VH2}) is motivated by condition (\ref{WB_B2}).
Space $\VH$ is closed subspace of $H^2(\Omega)$, and thus is also a Hilbert space.  
Thus, scalar product in $\VH$ can be defined in a following way 
\eqnsl{
\il{u}{v}_{\VH} = \il{\nabla^2 u}{ \nabla^2 v} + \il{\nabla u}{ \nabla v} + \il{u}{v} + \skk \gak \il{u \cdot n}{v \cdot n}_{\Gamma_k},
}{VH_scalar_product}
where $\lk$ was introduced in condition (\ref{WB_B2}). 
We see, that defined product $\il{\cdot}{\cdot}_{\VH}$ denotes equivalent norm to standard norm in $H^2(\Omega)$ due to trace theorem. \\
Next, space $ \VH $ is separable, as a subset of separable metric space, and thus has orthonormal basis $ \{ w_j \}_{j = 1}^\infty$. Furthermore, space $ \VH $ is a dense subset of $ V $. 
\noindent
Additionally, matrix $M_m$ is defined in a following way:
\eqnsl{
M_m = \left[ \il{w_i}{w_l} + \skk \gak \il{w_i \cdot n}{w_l \cdot n}|_{\Gamma_k} \right]_{i,l=1}^m.
}{Mm} 
For the avoidance of doubt, we introduce following norms:
\begin{itemize}
\item for $f \in L^2(\Omega)$: $\norm{f}_{2} = \left(\int_{\Omega} |f(x)|^2 dx\right)^{\frac{1}{2}} $,
\item for $f \in L^2([0,T])$: $\norm{f}_{2} = \left(\int_0^T |f(t)|^2 dt\right)^{\frac{1}{2}} $,
\item for $f \in L^2(0,T,L^2(\Omega)$: $\norm{f}_{2,2} = \left( \int_0^T \norm{f(t)}_2^2 dt \right)^{\frac{1}{2}}$,
\item for $f \in L^2(\Gamma_K)$: $\norm{f}_{2,\Gamma_K} = \left(\int_{\Gamma_K} |f(x)|^2 dS\right)^{\frac{1}{2}} $.
\end{itemize}

\section{Auxiliary lemma}
In this section, result concerning matrix $M_m$ will be shown. It will be helpful in next sections. 
\begin{lem} \label{matrixMm}
Matrix $M_m$ is invertible.
\end{lem}
\begin{proof}
We will show that $M_m$ has full rank, and thus is invertible. It will be done by contradiction. 

For purpose of this proof, we denote by $W^m = \spn \{ w_k \}_{k=1}^m$ and introduce scalar product in $W^m$ in a following way
\eqns{
((w_i,w_l)) = \il{w_i}{w_l} + \skk \gak \il{w_i \cdot n}{w_l \cdot n}|_{\Gamma_k}.
} 
Suppose, there is a row $\overline{k}$ of matrix $M_m$, that is linear combination of other rows, thus
\eqns{
\Big(\Big( w_{\overline{k}}, w_l \Big)\Big) - \sum_{j \in \{1,\dots , m \} \setminus \{ \overline{k} \} } \beta_j 
\Big(\Big(w_j, w_l\Big)\Big) 
= 
\Big(\Big(w_{\overline{k}} - \sum_{j \in \{1,\dots , m \} \setminus \{ \overline{k} \} } \beta_j w_j, w_l \Big)\Big) & 
= 0 \\
&\forall_{l=\{1,\dots , m \}}
}
for some $\beta_j$. This would imply, that $w_{\overline{k}}$ is a linear combination of other elements of $W^m$, which contradicts orthonormality $ \{ w_j \}_{j = 1}^\infty  $ in $\VH $ .
\end{proof}
\section{Weak formulation of problem and its motivation}
In this section, weak formulation of problem (\ref{stokes_sys}), (\ref{no_slip_cond}), (\ref{WB_B2}) will be derived. By formal multiplication of (\ref{stokes_sys}) by test function $w \in V$ and integration over $\Omega$, we obtain
\eqns{
\int_\Omega v_t w - \nu \int_\Omega \Delta v w +  \int_\Omega \nabla p w = \int_\Omega f w.
} 
By integration by parts, we get
\eqns{
\il{v_t}{w} + \nu \il{\nabla v}{\nabla w} - \nu \il{\nabla v \cdot n}{w}_{\partial \Omega} + \il{p \cdot n}{w}_{\partial \Omega} = \il{f}{w}.
}
Now, we can omit parts of boundary, that are not inlets/outlets due to (\ref{no_slip_cond})
\eqns{
\il{v_t}{w} + \nu \il{\nabla v}{\nabla w} + \skk \il{p \cdot n_k}{w}_{\Gamma_k} - \nu \il{\nabla v \cdot n_k}{w}_{\Gamma_k} = \il{f}{w}.
}
Using condition (\ref{WB_B2}), we get   
\eqns{
\il{v_t}{w} + \nu \il{\nabla v}{\nabla w} 
+ \skk \gak \il{n_k \left(v_t \cdot n_k \right)}{w}_{\Gamma_k}
+ \lk \il{n_k \left(v \cdot n_k \right)}{w}_{\Gamma_k}
+ \il{S_k n_k}{w}_{\Gamma_k}
= \il{f}{w}.
}
Finally, we obtain 
\eqnsl{
\il{v_t}{w} + \nu \il{\nabla v}{\nabla w} 
& + \skk \gak \il{v_t \cdot n_k}{w \cdot n_k}_{\Gamma_k} \\
& + \skk \lk  \il{v   \cdot n_k}{w \cdot n_k}_{\Gamma_k}
= \il{f}{w}
- \skk \il{S_k}{w \cdot n_k}_{\Gamma_k}
}{weak_form}
In order to properly define weak solution, we need some assumptions on functions $S_k$ and $f$:
\eqnsl{
S_k \in H^1(0,T)
}{assumSk}
and
\eqnsl{
f \in H^1(0,T,V).
}{assumf}
\begin{de} \label{weak_sol_def}
Let (\ref{assumSk}) and (\ref{assumf}) hold. Then, we say, that $ v \in H^1(0,T;V) $ is weak solution to problem (\ref{stokes_sys}), (\ref{no_slip_cond}), (\ref{WB_B2}), if for all $ w \in V $ (\ref{weak_form}) holds.  
\end{de}
We require additional time regularity from solution $v$ to give meaning to boundary scalar products (in a sense of trace theorem).
\section{Main theorem - existence of weak solution} 
\begin{theorem}\label{TW_MAIN}
Let $\Omega \subset \mathbb{R}^3$ be a bounded set, whose boundary $\partial \Omega$ is Lipschitz. 
Moreover, inlets/outlets of $\Omega$ are flat(see section 1).   
Additionally, we assume that
\begin{itemize}
\item $\forall_{k = 1, \dots, K} ~\gak , \lk > 0$,
\item $T > 0 $,
\item initial conditions: $v_0 \in \VH $
\end{itemize}   
and (\ref{assumSk}), (\ref{assumf}) hold. Then, there exists unique weak solution to problem (\ref{stokes_sys}), (\ref{no_slip_cond}), (\ref{WB_B2}) such that
\eqns{
v \in H^1(0,T; V).
}
\end{theorem}
\begin{proof}
Proof is organised in sections to make it more transparent. We will employ Galerkin method. Firstly, existence of solution to approximated system will be shown. Next, we will derive proper estimates, that will enable us to pass to a limit. Finally, we will show, that obtained solution is unique.
\subsection{Approximate system}
Let $ \{ f^m \}_{m = 1}^\infty $ and $ \{ S^m_k \}_{m = 1}^\infty $ be series of smooth functions, such that 
\eqnsl{
f^m \rightarrow f ~~~\text{in}~~~ H^1(0,T,V)
}{fConv} 
and
\eqnsl{
S^m_k \rightarrow S_k ~~~\text{in}~~~ H^1(0,T).
}{skConv} 
Let $ \{ w_i \}_{i = 1}^\infty $ be base of space $\VH $. We define approximate solution $v^m$ in a following way: let
\eqns{
v^m(t,x) = \sum_{i=1}^m g_i^m (t) w_i (x),
}
where functions $ g_i^m (t) $ solve system of equations
\eqnsl{
(v_t^m, w_l) + \nu (\nabla v^m, \nabla w_l) & + \skk \gamma_k (v_t^m \cdot n, w_l \cdot n)_{\Gamma_k} + \skk \lambda_k (v^m \cdot n, w_l \cdot n)_{\Gamma_k}  \\
& 
= (f^m, w_l) - \skk (S_k^m , w_l \cdot n)_{\Gamma_k}, ~~~ l = 1, \dots , m, 
}{stokes1}
with initial condition given by
\eqnsl{
v^m (0,x) = v_{0m}(x) = \sum_{k=1}^m \il{v_0}{w_k} w_k(x).
}{stokes2}
Construction of $v^m (0,x)$ implies, that
\eqnsl{
v^m(0) \rightarrow v_0 \text{~~~~in~~} H^2(\Omega).
}{initial_data_caovergence}
System (\ref{stokes1}) and (\ref{stokes2}) is system of ordinary differential equations. Problem can be reformulated in a  following way
\eqnsl{
\sum_{i=1}^m \il{w_i}{w_l} \frac{d}{dt} g_i^m(t) & + \nu \sum_{i=1}^m \il{\nabla w_i}{\nabla w_l} g_i^m(t) + \skk \sum_{i=1}^m \lk \il{ w_i \cdot n}{ w_l \cdot n}_{\Gamma_k} g_i^m(t) \\ 
&+ \skk \sum_{i=1}^m \gak \il{ w_i \cdot n}{ w_l \cdot n}_{\Gamma_k} \frac{d}{dt} g_i^m(t) = (f^m, w_l) - \skk \il{S^m_k}{w_l \cdot n}_{\Gamma_k} \\
& l = 1, \dots , m,
}{approx1}
with $ g_i^m (0) =  \il{v_0}{w_k}$. Due to lemma \ref{matrixMm} matrix  $M_m = \left[ \il{w_i}{w_l} + \skk \gak \il{w_i \cdot n}{w_l \cdot n}|_{\Gamma_k} \right]_{i,l=1}^m$ invertible, thus we can rewrite system in a following form
\eqnsl{
\frac{d}{dt} g_i^m + \sum_{l=1}^m \alpha_{il} \, g_{i}^m (t) = \sum_{l=1}^m \beta_{il} \left( (f^m, w_l) - \skk \il{S^m_k}{w_l \cdot n}_{\Gamma_k}  \right) \in C^\infty([0,T]) .
}{approx2_1}
Existence of solution $g_i^m(t)$ on time interval $[0,T]$ follows from classical theory of ODE. 
\subsection{Estimates}
In this section, we will show several estimates, that will enable us to pass to the limit in (\ref{stokes1}). Additionally, we have to obtain such regularity from estimates, that boundary terms will make sense (in sense of traces).
\subsubsection{$L^2(0,T;H^1(\Omega)) \cap L^\infty(0,T;L^2(\Omega)) $ control for $ \vm $}
Firstly, we will show some control of $\vm $. To do so, we multiply equation (\ref{stokes1}) by $g_l^m(t)$ and sum from $ 1 $ to $ m $. This effectively means, that equation (\ref{stokes1}) was tested by $\vm $
\eqns{
\il{\vmt}{\vm} + \nu \il{\gvm}{\gvm} &+ \skk \gak \il{\vmtn}{\vmn}_{\Gamma_k} + \skk \lk \il{\vmn}{\vmn}_{\pgk} \\
&= \il{f^m}{\vm} - \skk \il{\sk}{\vmn}_{\Gamma_k}.
}
Using Cauchy-Schwartz inequality, we obtain
\eqns{
\jd \ddt \ldwkw{\vm} + \nu \ldwkw{\gvm} & + \skk \frac{\gak}{2} \ddt \dwgam{\vmn} + \skk \lk \dwgamkw{\vmn}\\
&\leq \ldw{f^m} \ldw{\vm} + \skk \dwgam{\sk} \dwgam{\vmn}.
}
Now, we need to control r.h.s. side of above inequality. To do this, we first use Poincar\'e inequality  
$\il{f^m}{\vm} \le \ldw{f^m} \ldw{\vm} \le  C(p) \ldw{f^m} \ldw{\gvm} $. \\
Additionally, we use Young inequality to obtain 
\eqns{
\frac{1}{2} \ddt \ldwkw{\vm} &+ \nu \ldwkw{\gvm} + \skk \frac{\gak}{2} \frac{d}{dt} \dwgamkw{\vmn}  + \skk \lk \dwgamkw{\vmn} \\
&\leq C(p,\nu) \ldwkw{f^m} + \frac{\nu}{2} \ldwkw{\gvm}  + \skk \left( C(\lk)\dwgamkw{\sk} + \frac{\lk}{2} \dwgamkw{\vmn} \right).
}
Finally, we get 
\eqns{
\ddt \ldwkw{\vm} &+ \nu \ldwkw{\gvm} + \skk \gak \frac{d}{dt} \dwgamkw{\vmn}  + \skk \lk \dwgamkw{\vmn} \\
&\leq C(p,\nu) \ldwkw{f^m}  + \skk C(\lk)\dwgamkw{\sk} .
}
Integrating from $ 0 $ to $ t \in [0,T] $, we get 
\eqns{
&\ldwkw{\vm (t)}  + \nu \int_0^t{\ldwkw{\gvm}} + \skk \gak \dwgamkw{\vm(t) \cdot n} \\
&\leq C(\nu, p) \int_0^t{\ldwkw{f^m}} + \skk C(\lk) \int_0^t{\dwgamkw{\sk}} +\ldwkw{\vm (0)} + \skk \gak \dwgamkw{v^m (0) \cdot n}.
}
By convergence (\ref{initial_data_caovergence}), (\ref{fConv}), (\ref{skConv}) r.h.s. is bounded and following inequality holds
\eqns{
&\sup_{t \in [0,T]} \ldwkw{\vm (t)}  + \nu \ioT{\ldwkw{\gvm}} + \skk \gak \sup_{t \in [0,T]} \dwgam{\vm(t) \cdot n}
\leq
C(\nu, \lk, \norm{f}_{2,2}, \norm{S_k}_{2}, \norm{v_0}_{H^1} ).
}
\subsubsection{$L^\infty(0,T;L^2(\Omega))$ control for $\nabla \vm$}
Now, we will proceed to establish control of $ \nabla \vm $. By multiplying equation (\ref{stokes1}) by $\frac{d}{dt} g^m_l (t)$ and summing from $ 1 $ to $ m $, we  effectively testing equation (\ref{stokes1}) by $ \vmt $
\eqns{
\il{\vmt}{\vmt} + \nu \il{\gvm}{\gvmt} &+ \skk \gak \il{\vmtn}{\vmtn}_{\pgk} + \skk \lk \il{\vmn}{\vmtn}_{\pgk} \\
&= \il{f^m}{\vmt} - \skk \il{\sk}{\vmtn}_{\pgk}.
}
Applying H\"older and Young inequality, we get 
\eqns{
\ldwkw{\vmt} + \frac{\nu}{2} \frac{d}{dt} \ldwkw{\gvm} &+ \skk \gak \dwgamkw{\vmtn} + \skk \lk \frac{d}{dt} \dwgamkw{\vmn}\\
&\leq \jd \ldwkw{f^m} + \jd \ldwkw{\vmt} + \skk \left( C(\gak) \dwgamkw{\sk} + \frac{\gak}{2} \dwgamkw{\vmtn} \right). 
}
After simplification, we get
\eqns{
\frac{1}{2} \ldwkw{\vmt} + \frac{\nu}{2} \frac{d}{dt} \ldwkw{\gvm} &+ \skk \frac{\gak}{2} \dwgamkw{\vmtn} + \skk \lk \frac{d}{dt} \dwgamkw{\vmn} \\
&\leq \frac{1}{2} \ldwkw{f^m} + \skk C(\gak) \dwgamkw{\sk}.
}
Integrating from $ 0 $ to $ t \in [0,T] $, we get 
\eqns{
&\frac{1}{2} \iot{\ldwkw{\vmt}} + \frac{\nu}{2} \ldwkw{\gvm (t)} + \iot{\skk \frac{\gak}{2} \dwgamkw{\vmtn}} + \skk \lk \dwgamkw{v^m (t) \cdot n} \\
&\leq \frac{1}{2} \iot{\ldwkw{f^m}} +   \iot{\skk C(\gak) \dwgamkw{\sk}}  + \frac{\nu}{2} \ldwkw{\gvm (0)} + \skk \lk \dwgamkw{v^m (0) \cdot n}.
}
Again, we see that due to (\ref{initial_data_caovergence}), (\ref{fConv}) and (\ref{skConv}) r.h.s. side is bounded, and thus 
\eqns{
&\frac{1}{2} \iot{\ldwkw{\vmt}} + \frac{\nu}{2} \ldwkw{\gvm (t)} + \iot{\skk \frac{\gak}{2} \dwgamkw{\vmtn}} + \skk \lk \dwgamkw{v^m (t) \cdot n} \leq C^*.
}
In particular following hold
\eqnsl{
\sup_{t \in (0, T)}{\ldwkw{\gvm (t)}} \leq C^*
}{lInfNablaEst}
and 
\eqnsl{
\sup_{t \in (0, T)}{\dwgamkw{v^m (t) \cdot n}} \leq C^*.
}{pomVnabrzeg}
\subsubsection{$L^2(\Omega)$ and $L^2(\Gamma_k)$ control for $v_t(0)$ and $v_t(0) \cdot n$}
Now, we will show estimates on time derivatives in zero time. This will become useful in next subsection. To do this, we  test equation (\ref{stokes1}) by $ \vmt $  
\eqns{
\il{\vmt}{\vmt} + \nu \il{\gvm}{\gvmt} & + \skk \gak \il{\vmtn}{\vmtn}_{\pgk} + \skk \lk \il{\vmn}{\vmtn}_{\pgk}\\
&= \il{f^m}{\vmt} - \skk \il{\sk}{\vmtn}_{\pgk}.
}
After integration by parts, we get
\eqns{
\il{\vmt}{\vmt} - \nu \il{\dvm}{\vmt} &+ \nu \il{\gvm \cdot n}{\vmt}_{\dom} + \skk \gak \il{\vmtn}{\vmtn}_{\pgk} \\
&= \il{f^m}{\vmt} - \skk \il{\sk}{\vmtn}_{\pgk} - \skk \lk \il{\vmn}{\vmtn}_{\pgk}.
}
Like previously, we can employ H\"older and Young inequalities to obtain
\eqns{
&\ldwkw{\vmt} + \skk \gak \dwgamkw{\vmtn} \\
&\leq \frac{1}{2}\ldwkw{f^m} +  \frac{1}{2} \ldwkw{\vmt} + \skk \left( C(\gak) \dwgamkw{\sk} + \frac{\gak}{2} \dwgamkw{\vmtn} \right) + \nu \il{\dvm}{\vmt} \\
&- \nu \il{\gvm \cdot n}{\vmt}_{\dom} - \skk \lk \il{\vmn}{\vmtn}_{\pgk}.
}
Again, using H\"older and Young inequalities and (\ref{fConv}), (\ref{skConv}), we get
\eqns{
&\frac{1}{2} \ldwkw{\vmt} + \skk \frac{\gak}{2} \dwgamkw{\vmtn} \\
&\leq C(f, S_k, \gak) + \nu \il{\dvm}{\vmt} - \nu \il{\gvm}{\vmt}_{\dom} + \skk \left( \lk C(\varepsilon_1) \dwgamkw{\vmn} + \lk \varepsilon_1 \dwgamkw{\vmtn} \right).
}
We see, that due to (\ref{pomVnabrzeg}), term $ \dwgamkw{\vmn} $ is bounded
\eqns{
\frac{1}{2} \ldwkw{\vmt} + \skk \frac{\gak}{4} \dwgamkw{\vmtn} 
\leq 
C(f, S_k, \gak, C^*) + \nu \il{\dvm}{\vmt} - \nu \il{\gvm}{\vmt}_{\dom}.
}
We recall, that $\vm \in \VH $, and thus $\left( \gvm \cdot n, \vmt \right) = \left( \left( \gvm \cdot n \right) \cdot n, \vmtn \right)$
\eqns{
\frac{1}{2} \ldwkw{\vmt} + \skk \frac{\gak}{4} \dwgamkw{\vmtn} 
\leq 
C(f, S_k, \gak, C^*) + \nu \il{\dvm}{\vmt} - \nu \skk \il{\gvm \cdot n}{\vmt \cdot n}_{\Gamma_k}.
}
Now, we can use H\"older and Young inequalities to get 
\eqns{
\frac{1}{4} \ldwkw{\vmt} + \skk \frac{\gak}{8} \dwgamkw{\vmtn} 
\leq 
C(f, S_k, \gak, C^*) + C(\nu) \ldwkw{\dvm} + C(\nu) \dwgamkw{\gvm \cdot n}.
}
Take above inequality at time $ t = 0 $, we obtain 
\eqns{
\ldwkw{\vmt(0)} + \skk \gak \dwgamkw{\vmtn(0)} 
\leq 
C(f, S_k, \gak, C^*) + C(\nu) \ldwkw{\dvm(0)} + C(\nu) \dwgamkw{\gvm(0) \cdot n}.
}
Like previously, due to (\ref{initial_data_caovergence}) r.h.s. side is bounded. Finally, we obtain
\eqnsl{
\ldwkw{\vmt(0)} + \skk \gak \dwgamkw{\vmtn(0)} \leq C^{**}
}{estVmAtTime0}
\subsubsection{$L^2(0,T;H^1(\Omega)) \cap L^\infty(0,T;L^2(\Omega))$ estimates for $\vmt$ }
The aim of this section is to show higher order estimates for $v_t$. This part is crucial, because it will enable us to properly define boundary terms in weak solution(after passing to the limit). \\
Differentiation of equation (\ref{stokes1}) yields 
\eqns{
\il{\vmtt}{z} + \nu \il{\gvmt}{\nabla z} &+ \skk \gak \il{\vmttn}{z \cdot n}_{\pgk} + \skk \lk \il{\vmtn}{z \cdot n}_{\pgk} \\
&= \il{f_t^m}{z} - \skk \il{\frac{d}{dt} \sk}{z \cdot n}_{\pgk}.
}
By testing equation by $\vmt$, we get 
\eqns{
\il{\vmtt}{\vmt} + \nu \il{\gvmt}{\gvmt} &+ \skk \gak \il{\vmttn}{\vmtn}_{\pgk} + \skk \lk \il{\vmtn}{\vmtn}_{\pgk} \\
&= \il{f_t^m}{\vmt} - \skk \il{\frac{d}{dt} \sk}{\vmtn}_{\pgk}.
}
By using Poincar\'e, H\"older and Young inequalities, we get
\eqns{
\frac{1}{2} \frac{d}{dt} \ldwkw{\vmt} + &\nu \ldwkw{\gvmt} + \skk \frac{\gak}{2} \frac{d}{dt} \dwgamkw{\vmtn} + \skk \lk \dwgamkw{\vmtn} \\
&\leq C(p, \nu) \ldwkw{f_t^m} + \frac{\nu}{2} \ldwkw{\gvmt} + \skk \left( C(\lk) \dwgamkw{\frac{d}{dt} \sk} + \lk \dwgamkw{\vmtn} \right). 
}
After moving proper terms on l.h.s., we get
\eqns{
\frac{1}{2} \frac{d}{dt} \ldwkw{\vmt} + \frac{\nu}{2} \ldwkw{\gvmt} &+ \skk \frac{\gak}{2} \frac{d}{dt} \dwgamkw{\vmtn} \leq C(p, \nu) \ldwkw{f_t^m} + \skk C(\lk) \ldwkw{\frac{d}{dt} \sk}.
}
Integration from $ 0 $ to $t \in [0,T] $ yields
\eqns{
\ldwkw{\vmt (t)} & + \nu  \iot{\ldwkw{\gvmt}} + \skk \gak \dwgamkw{\vmt (t) \cdot n} \\ &\leq C(p, \nu) \iot{\ldwkw{f_t^m}} + \skk C(\lk) \iot{\ldwkw{\frac{d}{dt} \sk}} + \ldwkw{\vmt (0)} + \skk \gak \dwgamkw{\vmt (0) \cdot n}.
}
Using (\ref{estVmAtTime0}), (\ref{fConv}), (\ref{skConv}), we see that r.h.s. is bounded independently from $m$. Thus, we finally get
\eqns{
\sup_{t \in (0,T)}{\ldwkw{\vmt (t)}} & + \nu  \ioT{\ldwkw{\gvmt}} + \skk \gak \dwgamkw{\vmt (t) \cdot n} \\ &\leq 
C\left( \norm{f_t}_{L^2(0,T;L^2(\Omega)}, \norm{\ddt S_k}_{L^2(0,T)}, \lk, \nu, C^{**} \right).
}

\subsubsection{Conclusion of obtained estimates}
In previous sections, we obtained following estimates
\begin{align}
	\sup_{t \in (0, T)} \ldwkw{\vm (t)}  + \nu \int_0^T &\ldwkw{\gvm} \, dt + \sup_{t \in (0, T)} \skk \gak \dwgamkw{\vm (t) \cdot n} \leq C_1, \label{szac3a} \\
	&\sup_{t \in (0, T)}{\ldwkw{\gvm (t)}} \leq C_2,\\
	\sup_{t \in (0, T)} \ldwkw{\vmt (t)} + \nu \ioT{&\ldwkw{\gvmt}} + \sup_{t \in (0, T)} \skk \gak \dwgamkw{\vmt (t) \cdot n} \leq C_3.
	\label{szac3b}
\end{align}
\subsection{Passing to the limit}
Based on estimates (\ref{szac3a}) - (\ref{szac3b}), we can extract subsequence (which we again label $ m $), such that 
\begin{gather}
v^{m} \rightharpoonup v \quad \text{ in } H^1(0,T;H^1(\Omega)), \label{conv3a} \\
v^{m} \overset{\ast}{\rightharpoonup} v \quad \text{ in } L^\infty(0, T; H^1(\Omega)), \\
v^{m}_t \overset{\ast}{\rightharpoonup} v_t \quad \text{ in } L^\infty(0, T; L^2(\Omega)) \label{conv3b}
\end{gather}
Using above convergences, we can pass to the limit in (\ref{stokes1}). This system is fully linear, and thus we will refrain from showing detailed proof of passing to the limit. After passing to the limit, we obtain
\eqns{
\il{v_t}{w} + \nu \il{\nabla v}{\nabla w} 
& + \skk \gak \il{v_t \cdot n_k}{w \cdot n_k}_{\Gamma_k}
 + \skk \lk  \il{v   \cdot n_k}{w \cdot n_k}_{\Gamma_k} \\
& = \il{f}{w} 
- \skk \il{S_k}{w \cdot n_k}_{\Gamma_k}
 \forall ~~ w \in \VH
}
Using density of $V$ in $\VH$, we can lower assumptions on test functions
\eqnsl{
\il{v_t}{w} + \nu \il{\nabla v}{\nabla w} 
& + \skk \gak \il{v_t \cdot n_k}{w \cdot n_k}_{\Gamma_k}
 + \skk \lk  \il{v   \cdot n_k}{w \cdot n_k}_{\Gamma_k} \\
& = \il{f}{w} 
- \skk \il{S_k}{w \cdot n_k}_{\Gamma_k}
 \forall ~~ w \in V
}{weakSolProof}

\subsection{Uniqueness of solution}
Suppose, that $v^1$ and $v^2$ are two distinct weak solutions of problem (\ref{stokes_sys}), (\ref{no_slip_cond}), (\ref{WB_B2}), and thus they fulfil
\eqnsl{
\il{v^1_t}{w} + \nu \il{\nabla v^1}{\nabla w} 
& + \skk \gak \il{v^1_t \cdot n_k}{w \cdot n_k}_{\Gamma_k}
 + \skk \lk  \il{v^1   \cdot n_k}{w \cdot n_k}_{\Gamma_k} \\
& = \il{f}{w} 
- \skk \il{S_k}{w \cdot n_k}_{\Gamma_k}
 \forall ~~ w \in V
}{uniq1}
and 
\eqnsl{
\il{v^2_t}{w} + \nu \il{\nabla v^2}{\nabla w} 
& + \skk \gak \il{v^2_t \cdot n_k}{w \cdot n_k}_{\Gamma_k}
 + \skk \lk  \il{v^2   \cdot n_k}{w \cdot n_k}_{\Gamma_k} \\
& = \il{f}{w} 
- \skk \il{S_k}{w \cdot n_k}_{\Gamma_k}
 \forall ~~ w \in V.
}{uniq2}
After subtracting equations (\ref{uniq1}) and (\ref{uniq2}) and introducing function $ v^{D} = v^1 - v^2$, we get 
\eqns{
\il{v^D_t}{w} + \nu \il{\nabla v^D}{\nabla w} 
+ \skk \gak \il{v^D_t \cdot n_k}{w \cdot n_k}_{\Gamma_k}
 + \skk \lk  \il{v^D   \cdot n_k}{w \cdot n_k}_{\Gamma_k}
= 0 & ~~\forall ~ w \in V
}
and $ v^D(0) = 0 $. Function $ w = v^D $ is a good test function
\eqns{
\il{v^D_t}{v^D} + \nu \il{\nabla v^D}{\nabla v^D} 
+ \skk \gak \il{v^D_t \cdot n_k}{v^D \cdot n_k}_{\Gamma_k}
 + \skk \lk  \il{v^D   \cdot n_k}{v^D \cdot n_k}_{\Gamma_k}
= 0
}
Omitting non-negative terms, we get 
\eqns{
\il{v^D_t}{v^D} + \skk \gak \il{v^D_t \cdot n_k}{v^D \cdot n_k}_{\Gamma_k}
\le 0,
}
\eqns{
\ddt \ndk{v^D} + \skk \gak \ddt \norm{v^D \cdot n_k}_{2,\Gamma_k}^2
\le 0.
}
Integrating from $0$ to $t$, we get
\eqns{
\ndk{v^D(t)} + \skk \gak \norm{v^D(t) \cdot n_k}_{2,\Gamma_k}^2
& \le  \ndk{v^D(0)} + \skk \gak \norm{v^D(0) \cdot n_k}_{2,\Gamma_k}^2 \\
& = 0
}
Thus $v^D = 0$ and $v^1 = v^2$ a.e.. This concludes proof of uniqueness and whole proof.
\end{proof}

\section{Acknowledgement}
The authors wishes to thank the Warsaw University of Technology, where the paper was written, for financial support by Rector's grant for scientific clubs.

\end{document}

%% file: rysunek.pdf_tex
\begingroup%
  \makeatletter%
  \providecommand\color[2][]{%
    \errmessage{(Inkscape) Color is used for the text in Inkscape, but the package 'color.sty' is not loaded}%
    \renewcommand\color[2][]{}%
  }%
  \providecommand\transparent[1]{%
    \errmessage{(Inkscape) Transparency is used (non-zero) for the text in Inkscape, but the package 'transparent.sty' is not loaded}%
    \renewcommand\transparent[1]{}%
  }%
  \providecommand\rotatebox[2]{#2}%
  \newcommand*\fsize{\dimexpr\f@size pt\relax}%
  \newcommand*\lineheight[1]{\fontsize{\fsize}{#1\fsize}\selectfont}%
  \ifx\svgwidth\undefined%
    \setlength{\unitlength}{317.6265069bp}%
    \ifx\svgscale\undefined%
      \relax%
    \else%
      \setlength{\unitlength}{\unitlength * \real{\svgscale}}%
    \fi%
  \else%
    \setlength{\unitlength}{\svgwidth}%
  \fi%
  \global\let\svgwidth\undefined%
  \global\let\svgscale\undefined%
  \makeatother%
  \begin{picture}(1,0.8751685)%
    \lineheight{1}%
    \setlength\tabcolsep{0pt}%
    \put(0,0){\includegraphics[width=\unitlength,page=1]{rysunek.pdf}}%
    \put(0.55222653,0.70874506){\color[rgb]{0,0,0}\makebox(0,0)[lt]{\lineheight{1.25}\smash{\begin{tabular}[t]{l}\textbf{$\Gamma_0$}\end{tabular}}}}%
    \put(0.40464751,0.55526293){\color[rgb]{0,0,0}\makebox(0,0)[lt]{\lineheight{1.25}\smash{\begin{tabular}[t]{l}\textbf{$\Omega$}\end{tabular}}}}%
    \put(0.26718825,0.86560046){\color[rgb]{0,0,0}\makebox(0,0)[lt]{\lineheight{1.25}\smash{\begin{tabular}[t]{l}\textbf{$\Gamma_1$}\end{tabular}}}}%
    \put(0.07322728,0.50466441){\color[rgb]{0,0,0}\makebox(0,0)[lt]{\lineheight{1.25}\smash{\begin{tabular}[t]{l}\textbf{$\ldots$}\end{tabular}}}}%
    \put(-0.00098386,0.16228119){\color[rgb]{0,0,0}\makebox(0,0)[lt]{\lineheight{1.25}\smash{\begin{tabular}[t]{l}\textbf{$\Gamma_k$}\end{tabular}}}}%
    \put(0.36951585,0.00297003){\color[rgb]{0,0,0}\makebox(0,0)[lt]{\lineheight{1.25}\smash{\begin{tabular}[t]{l}\textbf{$\Gamma_{k_p}$}\end{tabular}}}}%
    \put(0.60420073,0.06284499){\color[rgb]{0,0,0}\makebox(0,0)[lt]{\lineheight{1.25}\smash{\begin{tabular}[t]{l}\textbf{$\Gamma_{k_p+1}$}\end{tabular}}}}%
    \put(0.90909749,0.26608233){\color[rgb]{0,0,0}\makebox(0,0)[lt]{\lineheight{1.25}\smash{\begin{tabular}[t]{l}\textbf{$\Gamma_K$}\end{tabular}}}}%
    \put(0.17084727,0.09180121){\color[rgb]{0,0,0}\makebox(0,0)[lt]{\lineheight{1.25}\smash{\begin{tabular}[t]{l}\textbf{$\ldots$}\end{tabular}}}}%
    \put(0.7814027,0.15504934){\color[rgb]{0,0,0}\makebox(0,0)[lt]{\lineheight{1.25}\smash{\begin{tabular}[t]{l}\textbf{$\ldots$}\end{tabular}}}}%
  \end{picture}%
\endgroup%